\documentclass[10pt]{article}
\usepackage{latexsym,amsmath,amsthm,verbatim,ifthen,amssymb}
\usepackage[latin1]{inputenc}
\usepackage[all]{xy}
\usepackage{graphicx}
\usepackage{epsfig}
\newdir{ >}{{}*!/-7pt/\dir{>}}
\newtheorem{theorem}{Theorem}
\newtheorem{corollary}[theorem]{Corollary}
\newtheorem{lemma}[theorem]{Lemma}
\newtheorem{proposition}[theorem]{Proposition}
\newtheorem{example}[theorem]{Example}
\newtheorem{definition}[theorem]{Definition}
\newtheorem{remark}[theorem]{Remark}

\def\nil{{\rm{nil}\hskip1pt}}
\def\cat{{\rm{cat}\hskip1pt}}
\def\wcat{{\rm{wcat}\hskip1pt}}
\def\secat{{\rm{secat}\hskip1pt}}

\def\TC{{\rm{TC}\hskip1pt}}
\def\wTC{{\rm{wTC}\hskip1pt}}

\def\Z{{\mathbb{Z}}}

\def\R{{\mathbb{R}}}
\def\C{{\mathbb{C}}}
\def\RP{{\mathbb{R}\mathbb{P}}}
\def\CP{{\mathbb{C}\mathbb{P}}}

\def\disfrac#1#2{{\displaystyle{\frac{#1}{#2}  }}}
\begin{document}

\title{Topological complexity and the homotopy cofibre of the diagonal map
\footnotetext{This work has been supported by FEDER funds through
the Ministerio de Educaci\'on y Ciencia project MTM2009-12081 and
``Programa Operacional Factores de Competitividade - COMPETE'',
 and by FCT -\emph{Fundação para a Ciência e a Tecnologia} through
 projects Est-C/MAT/UI0013/2011 and PTDC/MAT/0938317/2008.
}
}
\author{J.M. Garc\'{\i}a Calcines
 \footnote{Universidad de La Laguna,
%Facultad de Matem\'aticas, Departamento de Matem\'atica
 Fundamental, 38271 La Laguna, Spain. E-mail:
 \texttt{jmgarcal@ull.es}}
\,and L. Vandembroucq
 \footnote{Centro
 de Matem\'{a}tica, Universidade do Minho, Campus de Gualtar,
 4710-057 Braga, Portugal. E-mail: \texttt{lucile@math.uminho.pt}}
}

\date{\empty}

\maketitle

\begin{abstract}
In this paper we analyze some relationships between the
topological complexity of a space $X$ and the category of
$C_{\Delta _X},$ the homotopy cofibre of the diagonal map $\Delta
_X:X\rightarrow X\times X$. We establish the equality of the two
invariants for several classes of spaces including the spheres,
the H-spaces, the real and complex projective spaces and almost all
the (standard) lens spaces.
% As a consequence of this work and of a
%result by M. Farber, S. Tabachnikov and S. Yuzvinsky, we obtain
%that the immersion problem for the real projective space $\RP^n$
%is equivalent to the computation of the L.-S. category of
%$C_{\Delta _{\RP^n}}.$
\end{abstract}

% \vspace{0.5cm}
% \noindent{2000 \textit{Mathematics Subject Classification} : 55M30.}\\
% \noindent{\textit{Keywords} : Sectional category, Topological complexity.}
% \vspace{0.2cm}

\section*{Introduction}

The topological complexity of a space $X,$ $\TC (X)$ is the
sectional category (or Schwarz genus) of the end-points evaluation
fibration $\pi _X:X^I\rightarrow X\times X,$ $\pi _X(\alpha
)=(\alpha (0),\alpha (1)).$ This homotopical invariant was defined
by M. Farber \cite{Far}, \cite{Far2}, giving a topological
approach to the robot motion planning problem. In robotics if one
regards the topological space $X$ as the configuration space of a
mechanical system, the motion planning problem consists of
constructing a program or a devise, which takes pairs of
configurations $(A,B)\in X\times X$ as an input and produces as an
output a continuous path in $X,$ which starts at $A$ and ends at
$B.$ Broadly speaking, $\TC (X)$ measures the discontinuity of any
motion planner in the space. Further developments of the
topological complexity have proved to be a very interesting
homotopical invariant. It not only interacts with robotics but
also with deep problems arising in algebraic topology. In this
sense it is shown that the problem of computing the number
$\TC(\RP^n)$ of the $n$-th real projective space is equivalent to
a classical problem of manifold topology (see \cite{F-T-Y}) which
asks for the minimal dimension of the Euclidean space $N$ such
that there exists an immersion $\RP^n\rightarrow \mathbb{R}^N.$
More exactly M. Farber, S. Tabachnikov and S. Yuzvinsky proved the
following theorem

\begin{theorem}[\cite{F-T-Y}]\label{TC-immersion}{\rm For $n\neq 1,3,7$ the (normalized)
topological complexity of the $n$-dimensional real projective
space $\RP^n$ coincides with the least integer $k$ such that
$\RP^n$ can be immersed in $\R^k$.}
\end{theorem}

%The topological complexity of lens spaces is also shown to be , as
%shown by J. Gonz\'{a}lez in \cite{G}.

The purpose of this paper is to compare the topological complexity
of a space to the classical Lusternik-Schnirelmann of the homotopy
cofibre $C_{\Delta _X}$ of the diagonal map $\Delta
_X:X\rightarrow X\times X$. This question is suggested by our
previous work \cite{Weaksecat} where, by analogy to the weak
category (\wcat) of Berstein and Hilton, we introduced a ``weak"
version of the topological complexity (\wTC) and proved that, for
any space $X$, one has $\wTC(X)=\wcat(C_{\Delta _X}).$ However,
the question of the equality in the non weak case seems to be much
harder than in the weak case.  We obtain inequalities between the
two invariants under some restrictions and we establish the
equality in several particular cases. More precisely, in Section
3, we prove that $\cat(C_{\Delta_X})\leq \TC(X)$ holds when $X$ is
a $H$-space or when $X$ is a $(q-1)$-connected CW-complex
satisfying the condition $\dim(X)\leq q(\TC(X)+1)-2$. We then establish $\cat(C_{\Delta_X})=\TC(X)$ for several examples
including the spheres, the compact orientable surfaces, the
H-spaces and the complex projective spaces. Such formula also
holds for real projective spaces and (standard) lens spaces (with
a slight restriction) but these cases require more material and
are developed in Section 4. Notice that the identity $\TC(\RP^
n)=\cat(C_{\Delta_{\RP^n}})$ together with the result by M.
Farber, S. Tabachnikov and S. Yuzvinsky (Theorem
\ref{TC-immersion} above), shows that the immersion problem for
the real projective space $\RP^n$ is equivalent to the computation
of the L.-S. category of $C_{\Delta _{\RP^n}}.$ The main result of
Section 2 is a majoration of the form $\TC(X)\leq \cat(Y)$ where
$Y$ is the target of a map $f:X\to Y$ satisfying some condition.
This result is in particular useful in order to prove
$\cat(C_{\Delta_X})=\TC(X)$ for H-spaces, real projective spaces
and lens spaces.
 We finally note that, along this paper, we work
with the normalized version of all the invariants related to the
topological complexity and Lusternik-Schnirelmann category. The
definition of these invariants as well as some basic constructions
are recalled in Section 1.
%We also note that the auxiliary
%results we obtain could be useful in the computation of $\TC(X)$.

\section{Preliminary notions and results.}
The aim of this section is to recall the most important notions
and results that will be used in this paper. We assume that the
reader is familiarized with the notions of homotopy commutative
diagrams, homotopy pushouts and homotopy pullbacks; we refer the
reader to \cite{M}, \cite{D} or \cite{B-K}. We also point out that
the category in which we shall work throughout this paper is the
category of well-pointed compactly generated Hausdorff spaces.
Therefore all categorical constructions are carried out in this
category.

\subsection{Join of maps}\label{joins}

Recall that, given any pair of maps
$A\stackrel{f}{\longrightarrow}C\stackrel{g}{\longleftarrow }B$,
the \textit{join of} $f$ \textit{and} $g,$ $A*_C B,$ is the
homotopy pushout of the homotopy pullback of $f$ and $g$

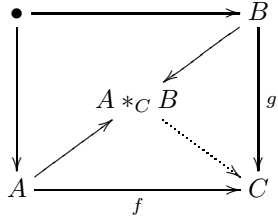
$$\xymatrix@C=0.7cm@R=0.7cm{ {\bullet } \ar[rr] \ar[dd] & & {B} \ar[dl] \ar[dd]^g \\
 & {A*_C B} \ar@{.>}[dr] & \\ {A} \ar[ur] \ar[rr]_f & & {C} }$$
where the dotted arrow is the so-called co-whisker map, induced by
the weak universal property of homotopy pushouts. When $C=*$ we
use the notation $A*B.$ Notice that the map $A*_C B\to C$ is only
defined up to homotopy equivalence. Any representative is homotopy
equivalent to the canonical co-whisker map obtained by considering
first the standard homotopy pullback of $f$ and $g$ and then the
standard homotopy pushout of the projections on $A$ and $B$. When
$f$ (or $g$) is a fibration our preferred representative of the
join map $A*_C B\to C$ will be given by the following explicit
construction: we first take the honest pullback of $f$ and $g$
(which is a homotopy pullback) and then we take the standard
homotopy pushout of the projections $A\times_C B\to A$ and
$A\times_C B\to B$. In this way we obtain the following map:
$$\begin{array}{rcl}
A\ast_C B=A\amalg (A\times_C B\times [0,1]) \amalg B/\sim &\to&C\\
&&\\
\langle a,b,t\rangle & \mapsto & f(a)=g(b)\\
\end{array}$$
where $\sim$ is given by $(a,b,t)\sim a$ if $t=0$ and $(a,b,t)\sim
b$ if $t=1$. When $f$ and $g$ are both fibrations, this explicit
construction coincides, up to homeomorphism, with the notion of
sum of two fibrations used by Schwarz \cite{Sch} and is a
fibration whose fibre is the ordinary join of the fibres. Given a
fibration $p:E\to B$ we will call \emph{$n$-fold fiber join} of
$p$ and denote it by $j^n_p:*^n_B E\rightarrow B$ the fibration
obtained by taking the join of $n+1$ copies of $p$. More
precisely, we set $j^0_p=p,$ $*^0_B E=E$ and $j^n_p:*^n_B
E\rightarrow B$ is the join of $j^{n-1}_p$ and $p$. If we denote
by $F$ the fibre of $p$, the fibre of $j^n_p$, denoted by $*^nF,$
is the join of $n+1$ copies of $F$. Using the above explicit
construction the $n$-fold fiber join is natural in the sense that
any commutative diagram
$$\xymatrix{
{E'} \ar[r]^{f'} \ar@{>>}[d]_{p'} & {E} \ar@{>>}[d]^p \\
{B'} \ar[r]_{f} & {B} }$$ where $p$ and $p'$ are fibrations,
induces for each $n\geq 0$ a commutative diagram as follows
$$\xymatrix{
{*^n_{B'}E'} \ar[r] \ar@{>>}[d]_{j^n_{p'}} & {*^n_{B}E} \ar@{>>}[d]^{j^n_p} \\
{B'} \ar[r]_{f} & {B} }$$
Moreover, as a consequence of the Join Theorem
(\cite{D}), if the first square is a (homotopy) pullback then so is the second one.

%\begin{lemma}\label{pullback-join}{\rm Consider a (homotopy) pullback of a fibration
%$p:E\twoheadrightarrow B$ along a map $f:B'\rightarrow B$
%$$\xymatrix{
%{E'} \ar[r]^{f'} \ar@{>>}[d]_{p'} & {E} \ar@{>>}[d]^p \\
%{B'} \ar[r]_{f} & {B} }$$ Then, for each $n\geq 0$ there is a
%(homotopy) pullback
%$$\xymatrix{
%{*^n_{B'}E'} \ar[r] \ar@{>>}[d]_{j^n_{p'}} & {*^n_{B}E} \ar@{>>}[d]^{j^n_p} \\
%{B'} \ar[r]_{f} & {B} }$$  }
%\end{lemma}
%
%\noindent \textbf{Remark.} We point out that there is a similar
%result replacing \emph{(homotopy) pullback} for just
%\emph{commutative square} in the statement of the above lemma.

%\begin{proposition}[\textbf{Gluing Lemma}]{\rm
% Suppose we have a homotopy commutative cube
%$$\xymatrix@!0{
%  & {\bullet } \ar[dl] \ar[rr] \ar'[d]^{f_1}[dd]
%      &  & {\bullet } \ar[dd]^{f_2} \ar[dl]       \\
%  {\bullet } \ar[rr]\ar[dd]^{f_3}
%      &  & {\bullet } \ar[dd]^(.3){f_4} \\
%  & {\bullet } \ar[dl] \ar'[r][rr]
%      &  & {\bullet }   \ar[dl]             \\
%  {\bullet } \ar[rr]
%      &  & {\bullet }         }$$
%\noindent where the top and bottom faces are homotopy pushouts and
%$f_1,f_2,f_3$ are homotopy equivalences. Then, $f_4$ is also a
%homotopy equivalence.}
%\end{proposition}

%\begin{proposition}[\textbf{Join Theorem}]{\rm
% Consider a homotopy commutative diagram where the squares are
% homotopy pullbacks
%$$\xymatrix{
%{A} \ar[d] \ar[r] & {B} \ar[d] \ar[r] & {C} \ar[d] \ar[l] \\
%{A'} \ar[r] & {B'} & {C'} \ar[l] }$$ Then this diagram induces a
%homotopy pullback
%$$\xymatrix{
%{A*_{B}C} \ar[d] \ar[r] & {B} \ar[d] \\
%{A'*_{B'}C'} \ar[r] & {B'} }$$ }
%\end{proposition}

\subsection{L.-S. category and topological complexity.}\label{LS-cat and TC}
%In this paper we shall deal with the Lusternik-Schnirelmann
%category.
The (normalized) \emph{Lus\-ter\-nik-Schnirelmann category} of a space $X$, or
\emph{L.-S. category} for short, is the smallest integer $n$ such
that there is an open cover of $X$ of $n+1$ elements, each of
which is contractible in X. We write $\cat (X)=n$ if such integer
exists; otherwise we state $\cat (X)=\infty .$ A useful lower
bound for $\cat (X)$ is the cup-length of $X,$ $\mbox{cup}(X).$
This is just the index of nilpotency of the reduced cohomology
$\tilde H^*(X)$ with coefficients in any ring $R.$

On the other hand the (normalized) \emph{topological complexity}
of $X,$ $\TC (X)$, is the smallest integer $n$ such that there is
an open cover of $X\times X$ of $n+1$ elements, on each of which
the fibration $\pi _X:X^I\rightarrow X\times X,$ $\pi _X(\alpha
)=(\alpha (0),\alpha (1))$ admits a (homotopy) section. In
\cite{Far}, Farber established that, for any path-connected space
$X$, one has
\begin{equation} \label{Farberinequalities}
\cat (X)\leq \TC(X)\leq \cat(X\times X)\leq 2\cat(X)
\end{equation}
 Also, if $\nil \ker \hspace{3pt}\cup $
denotes the index of nilpotency of $\ker \hspace{3pt}\cup $, where
$\cup :\tilde H^*(X)\otimes \tilde H^*(X)\rightarrow \tilde
H^*(X)$ is the usual cup-product of the reduced cohomology with
coefficients in a field \textbf{K}, then Farber proved that $\nil
\ker \hspace{3pt}\cup $ is a lower bound of $\TC (X).$

L.-S. category and topological complexity are both special cases
of the notion of \emph{sectional category} (or \emph{genus} as
first introduced by A. Schwarz in \cite{Sch}) of a fibration
$p:E\twoheadrightarrow B$. This invariant, written $\secat (p),$
is the smallest integer $n$ such that there is an open cover of
$B$ of $n+1$ elements, on each of which the fibration $p$ admits a
(homotopy) section. We have $\cat(X)=\secat(ev:PX\to X)$ and
$\TC(X)=\secat(\pi _X:X^I\rightarrow X\times X)$. Here, $PX\subset
X^I$ is the space of paths beginning at the base point $*$ and
$ev=ev_1$ is the evaluation map at the end of the path.

The sectional category of a fibration $p:E\twoheadrightarrow B$ can
be characterized in terms of the $n$-fold fiber join of $p$ (see \cite{J}):

\begin{theorem}{\rm
If $B$ is a paracompact space, then \secat$(p)\leq n$ if and only
if $j^n_p:*^n_B E\rightarrow B$ admits a homotopy section.}
\end{theorem}

In particular, considering $p=ev:PX\to X$ we have
$$\cat(X)\leq n \quad \Leftrightarrow \quad j^n_{ev}:*^n_X PX\rightarrow
X \mbox{ admits a homotopy section;}$$ and considering $p=\pi:X^
I\to X\times X$ we have $$\TC(X)\leq n \quad \Leftrightarrow \quad
j^n_{\pi}:*^n_{X\times X} X^ I\rightarrow X\times X \mbox{ admits
a homotopy section.}$$

\noindent Both these fibrations have as fibre the join $*^n\Omega
X$, where $\Omega X$ denotes the loop space of $X.$ Note that the
fibration $j^n_{ev}:*^n_X PX\rightarrow X$ is equivalent to the
so-called \emph{$n$-th Ganea fibration} $g_n:G_n(X)\to X.$ Recall
that $g_0=ev:PX\to X$ and $g_n$ is the fibration associated to the
join of $g_{n-1}$ with the trivial map $\ast \to X.$ These
fibrations fit in a homotopy commutative diagram of the following
form
$$\xymatrix{
{F_0(X)} \ar[d] & {F_1(X)} \ar[d] & {F_2(X)} \ar[d] & {...} \\
{G_0(X)} \ar@{>>}[d]_{g_0} \ar[r] & {G_1(X)} \ar@{>>}[d]_{g_1}
\ar[r] & {G_2(X)} \ar@{>>}[d]_{g_2} \ar[r] & {...} \\
{X} \ar@{=}[r] & {X} \ar@{=}[r] & {X} \ar@{=}[r] & {...} }$$ where
each column is a fibration and for each $n$ $F_n(X)\to G_n(X)\to
G_{n+1}(X)$ is a homotopy cofibration sequence, the so-called
\emph{fibre-cofibre construction}. Numerous lower bounds that
approximate the L.-S. category have been defined in terms of Ganea
fibrations. For instance, in this work
we will use the following lower bound, denoted by $\sigma^1
\cat(X)$:
$$\sigma^1\cat(X)\leq n \quad \Leftrightarrow \quad \Sigma g_n
\mbox{ admits a homotopy section}$$ \noindent where $\Sigma $
denotes the usual suspension functor. This invariant can also be
characterized in the following way (as the weak category of Ganea)
$$\sigma^1\cat(X)\leq n \quad \Leftrightarrow \quad X\to
X\cup_{g_n}CG_n(X) \mbox{ is homotopically trivial}$$ and is
greater than or equal to most of approximations of the L.S.
category. In particular $\sigma^1 \cat(X)$ is greater than or
equal to the weak category in the sense of Berstein-Hilton
($\wcat$), to the $\sigma$-category and to the Toomer invariant.
We refer the reader to \cite{C-L-O-T} for more details and more
information on L.-S. category and its approximations.

\section{Some majorations of $\TC(X).$}

As we have recalled in (\ref{Farberinequalities}) of the previous
section, there is an inequality
$$\TC(X)\leq \cat(X\times X).$$
In many cases, this inequality turns out to be strict. For
instance, for $S^{2n+1}$ an odd-dimensional sphere,
$\TC(S^{2n+1})=1$ while $\cat(S^{2n+1}\times S^{2n+1})=2.$ One of
the objectives of many works on topological complexity (see, for
instance \cite{Farber-Grant}, \cite{Grant}) is to provide
alternative ways to get an upper bound for $\TC$. The following
theorem gives two results in this direction. The first item is an
adaptation for $\TC$ of a result established in \cite{F-H-T} for
the L.-S. category. If $X$ is a space with base point $*$, then
$i_1,i_2:X\rightarrow X\times X$ will denote the inclusions
$i_1(x)=(x,*),$ $i_2(x)=(*,x)$ and $p_1,p_2:X\times X\rightarrow
X$ the projections $p_1(x,y)=x$ and $p_2(x,y)=y$.

\begin{theorem}\label{Majorations of TC}{\rm
Let $f:X\to Y$ be a map between $(q-1)$-connected CW-complexes
($q\geq 1$) which is an $r$-equivalence.
 \begin{enumerate}
 \item[(a)] If $2\dim X\leq r+q\TC(Y)-1,$ then
$\TC(X)\leq \TC(Y)$.
\item[(b)] If there exists a map $g:X\times X\to Y$ such that $f=gi_1$ and $g\Delta_X\simeq *$ and if $2\dim X\leq
r+q\cat(Y)-1,$ then $\TC(X)\leq \cat(Y)$.\\
\end{enumerate}}
\end{theorem}

\begin{proof}
We denote by $f^I:X^I\to Y^I$ and $\Omega f:\Omega X\to \Omega Y$
the maps induced by $f$. For the proof of $(a)$ suppose that $\TC
(Y)\leq n$. We construct the $n$-fold join (recall that this means
the join with $n+1$ factors) of the fibrations $\pi_X:X^I\to
X\times X$ and $\pi_Y:Y^I\to Y\times Y$ and consider the pullback
$Q$ of the fibration $\ast^n_{Y\times Y} Y^I\to Y\times Y$ along
the map $f\times f$. We thus have the following commutative
diagram:
$$\xymatrix{
\ast^n\Omega X\ar[d]\ar[rr]^{\ast^n\Omega f} && \ast^n\Omega Y\ar[d]\ar@{=}[r]
& \ast^n\Omega Y\ar[d]\\
\ast^n_{X\times X}X^I\ar[rr]\ar[dr]&& Q\ar[dl] \ar[r]&\ast^n_{Y\times Y} Y^I \ar[d]\\
&X\times X \ar[rr]^{f\times f} && Y\times Y. }$$ By the homotopy
exact sequence of a fibration one can check that the map
$\ast^n_{X\times X}X^I\to Q$ has the same connectivity as
$\ast^n\Omega f.$ We point out that $Q$ is connected. Indeed, when
$n\geq 1$, this is certainly true as the fibre $\ast^n\Omega Y$ is
connected; when $n=0$ we have $\TC(Y)=0$ so $Y$ is contractible
and therefore $Q\simeq X\times X$ is connected.

The map $\ast^n\Omega f$ is an $(r+qn-1)$-equivalence since $f$ is
an $r$-equivalence (see, for example, \cite{F-H-T}). Now,
considering $\TC (Y)\leq n$ and the pullback property we have that
the fibration $Q\to X\times X$ admits a section $\sigma .$
Finally, since the integer $\dim(X\times X)$ is less than or equal
to the connectivity of $\ast^n\Omega f$, the section $\sigma$
lifts in a homotopy section of $\ast^n_{X\times X}X^I\to X\times
X$ and therefore $\TC(X)\leq n.$

Next we prove $(b)$. Since $g\Delta _X$ is homotopically trivial
we can consider the following diagram in which the left-hand
square is a pullback (and a homotopy pullback) and the right-hand
square is commutative.
$$\xymatrix{PX \ar[r]^{\bar{i_1}}\ar[d]_{ev_{1}}& X^I \ar[d]^{\pi}
\ar[r]^{\bar g} & PY \ar[d]^{ev_1} \\
X \ar[r]_{i_1}& X\times X \ar[r]_g & Y. }$$ By taking the fibres
of the vertical maps, we see that the map $\hat g:\Omega X\to
\Omega Y$ induced by $g$ and $\bar g$ between the fibres is equal
to the map $\hat f$ induced by $f=gi_1$ and $\bar f=\bar g
\bar{i_1}$. On the other hand, it is possible to construct a
fibrewise homotopy between $\bar f$ and $Pf:PX\to PY$ so that
$\hat f$ is homotopic to $\Omega f$ and hence $\hat g\simeq \Omega
f$. Now we suppose that $\cat (Y)\leq n$ and, following the same
strategy as in the proof of $(a)$, we construct the $n$-fold join
of the fibrations $\pi:X^I\to X\times X$
 and $ev_1:PY\to Y$. Then we take the pullback $Q$ of the
fibration $*^n_Y PY\to Y$ along the map $g$ and we obtain:
$$\xymatrix{
\ast^n\Omega X\ar[d]\ar[rr]^{\ast^n\hat g} && \ast^n\Omega Y\ar[d]\ar@{=}[r]
& \ast^n\Omega Y\ar[d]\\
\ast^n_{X\times X}X^I\ar[rr]\ar[dr]&& Q\ar[dl] \ar[r]&\ast^n_Y PY \ar[d]\\
&X\times X \ar[rr]^g && Y. }$$ Since $\ast^n\hat g\simeq
\ast^n\Omega f$ and $f$ is an $r$-equivalence, the map
$\ast^n_{X\times X}X^I\to Q$ is an $(r+qn-1)$-equivalence. Now
from $\cat (Y)\leq n$ and the condition on $\dim(X\times X)$ we
deduce the existence of a homotopy section of $\ast^n_{X\times
X}X^I\to X\times X,$ that is, $\TC(X)\leq n.$
\end{proof}

Here are some applications of Theorem
\ref{Majorations of TC}.

\begin{example}{\rm
Let $\RP^n=S^n/\mathbb{Z}_2$ be the real $n$-dimensional
projective space. If $n\leq k,$ then $\TC(\RP^n)\leq \TC(\RP^k)$.}
\end{example}
This result already appears in \cite{F-T-Y}. Here we give a proof,
which does not depend on Theorem \ref{TC-immersion}.

\begin{proof}
The inclusion $f:\RP^n\hookrightarrow \RP^k$ is an
$n$-equivalence. If $k\geq n+1,$ then we have $\TC(\RP^k)\geq
\cat(\RP^k)=k\geq n+1$. Then, taking $q=1$, we have
$2\dim(\RP^n)=2n\leq n+k-1$ and Theorem \ref{Majorations of TC}(a) permits
to conclude that $\TC(\RP^n)\leq \TC(\RP^k)$.
\end{proof}

\begin{example}{\rm
Similarly, consider $L^{2n+1}_p=S^{2n+1}/\mathbb{Z}_p$ the
$(2n+1)$-dimensional lens space. If $n\leq k,$ then
$\TC(L^{2n+1}_p)\leq \TC(L^{2k+1}_p)$.}
\end{example}

The proof is analogous to that of the previous example using the well known fact that
$\cat(L^{2k+1}_p)=2k+1$ (see, for instance, \cite{C-L-O-T}). Actually, using Lemma \ref{cat-lens} below, one also can prove that
$\TC((L^{\infty }_p)^{(n)})\leq \TC((L^{\infty }_p)^{(k)}),$ for
$n\leq k$, where $(L^{\infty }_p)^{(m)}$ denotes the $m$-skeleton
of the infinite lens space $L^{\infty }_p.$\\

The following example shows how Theorem \ref{Majorations of TC} can
help to compute the topological complexity of a space in a
concrete situation:

\begin{example}{\rm
Let $X=S^3\cup_{\alpha}e^7$ where $\alpha$ is the Blakers-Massey
element of $\pi_6(S^3)$. Then we have that $\TC(X)=3$.}
\end{example}

\begin{proof}
Recall that the symplectic group $Sp(2)$ admits a cellular
decomposition of the form $S^3\cup_{\alpha}e^7\cup_{\beta}e^{10}$.
Then $X=S^3\cup_{\alpha}e^7$ is a subcomplex of $Sp(2)$ such that
the inclusion $f:X\hookrightarrow Sp(2)$ a $9$-equivalence. By
\cite{Schwe} we know that $\cat(Sp(2))=3,$ and Farber showed in
\cite{Far2} that $\TC(G)=\cat(G),$ for any Lie group $G.$
Therefore $\TC(Sp(2))=3$. Theorem \ref{Majorations of TC}(a) with
$q=3$ permits us to obtain the inequality $\TC(X)\leq 3.$  On the
other hand, in \cite{Weaksecat} we proved that $\wTC(X)= 3$. Since
$\wTC(X)\leq \TC(X)$ we conclude that $\TC(X)= 3.$
\end{proof}

In the above example we may also obtain the result from the
statement (b) of Theorem \ref{Majorations of TC} by considering
the map
$$g:X\times X \hookrightarrow Sp(2)\times Sp(2) \rightarrow Sp(2)$$
given by $g(x,y)=xy^{-1}$. In this way we obtain $\TC(X)\leq
\cat(Sp(2))=3$. In the following statement, we generalize this
process to subcomplexes $X$ of a CW-complex $H$ that it is also an
$H$-space. When $X=H$ we will obtain $\TC(H)\leq \cat(H)$ which,
together with the general inequality $\cat\leq \TC$ gives
$\TC(H)=\cat(H)$. We note that this fact is also proved in
\cite{L-S} and generalizes the result by Farber that $\TC(G)=
\cat(G)$ for any Lie group $G$ (\cite{Far2}).

\begin{corollary}\label{TC(X)<=cat(Y)-group}{\rm
Let $H$ be a connected CW-complex which is an $H$-space (with $*$
as unit element). Then
\begin{enumerate}
\item If
$X$ is a $(q-1)$-connected subcomplex of $H$ such that the
inclusion $X\hookrightarrow H$ is an $r$-equivalence and $2\dim
X\leq r+q\cat(H)-1,$ then $\TC(X)\leq \cat(H)$.
\item $\TC(H)=\cat(H).$
\end{enumerate} }
\end{corollary}

\begin{proof}
Since $H$ is a $H$-space that it is a connected CW-complex, the
shear map $\phi:H\times H\to H\times H$ given by
$\phi(x,y)=(x,xy)$ is a homotopy equivalence (\cite[p.
461]{Whitehead}). Let $\psi$ be a homotopy inverse of $\phi$ which
preserves the base point. The map $j:H\to H$ given by $j(x)=p_2
\psi(x,\ast)$ is a right inverse of the identity (recall that $p_2$
denotes the projection onto the second factor). This means that the
composite
$$\xymatrix{
H \ar[r]^{\Delta} & H\times H \ar[r]^{id\times j} & H\times H
\ar[r]^{\mu} & H }$$ where $\mu$ is the multiplication, is
homotopically trivial (\cite[p. 119]{Whitehead}). We note that the
(homotopy) associativity is not required for this fact and that
$j(*)=*$. Now consider the map $g:X\times X\to H$ given by
$g(x,y)=\mu(x,j(y))$. We have $g\Delta_H\simeq *$ and the map
$f=gi_1$ is the inclusion $X\hookrightarrow H$. The hypothesis on
the inclusion together with the statement (b) of Theorem
\ref{Majorations of TC} gives $\TC(X)\leq \cat(H).$ If $X=H$ the
map $f$ is the identity and we obtain $\TC(H)\leq \cat(H)$. The
general inequality $\cat\leq \TC$ allows us to conclude that
$\TC(H)=\cat(H).$
\end{proof}

We finally note that Theorem \ref{Majorations of TC}(b) also
permits us to recover the following result by Farber, Tabachnikov and Yuzvinky.

\begin{theorem}(\cite{F-T-Y}) {\rm
Consider integers $k>n.$ If there is an axial map $g:\RP^n\times
\RP^n\to \RP^k$ of type $(n,k)$, that is a map $\RP^n\times
\RP^n\to \RP^k$ whose restrictions to $*\times \RP^n$ and
$\RP^n\times *$ are homotopic to the inclusion
$\RP^n\hookrightarrow \RP^k$, then $\TC(\RP^n)\leq k$.}
\end{theorem}

\begin{proof} The axial property together with $k>n$ implies that
$g\Delta_X\simeq *$ (see \cite{F-T-Y} for more details). On the
other hand, the composite $f=gi_1$ is homotopic to the inclusion
$\RP^n\hookrightarrow \RP^k$, so it is an $n$-equivalence.
Finally, since $\cat(\RP^k)=k$ we have $2\dim(\RP^n)=2n\leq n+k-1$
and, by Theorem \ref{Majorations of TC}(b), we can conclude that
$\TC(\RP^n)\leq k$.
\end{proof}

\section{$\TC(X)$ and $\cat(C_{\Delta _X}).$}

For a space $X$, we shall denote by $C_{\Delta _X}$ the homotopy
cofiber of the diagonal map $\Delta_X:X\to X\times X$, that is
$C_{\Delta_X}=X\times X \cup_{\Delta_X} CX.$ We also denote by
$\alpha:X\times X \to C_{\Delta_X}$ the induced map. We observe
that, for \emph{locally equiconnected} spaces (i.e., those spaces
for which $\Delta _X$ is a cofibration) $C_{\Delta _X}$ is, up to
homotopy equivalence, the quotient space $(X\times X)/\Delta _X(X)$. Note that the class of locally equiconnected spaces is
not restrictive. For instance, the CW-complexes and the metrizable
topological manifolds fit on such a class of spaces
(\cite{D-E,Du}).

In this section we will analyze the relationship between $\TC(X)$
and $\cat(C_{\Delta _X}).$ Then we give some first examples of
when the equality holds.

\subsection{General results on $\TC(X)$ and $\cat(C_{\Delta _X}).$}

In some examples, that will be developed later, Theorem
\ref{Majorations of TC} will permit us to establish the inequality
$\TC(X)\leq \cat(C_{\Delta_X})$. However, in the general case, we
have to consider a rather strong condition on the dimension of
$X$. Namely, taking $g=\alpha: X\times X\to C_{\Delta_X}$ and
using the fact that the map $gi_1$ is a $(2q-1)$-equivalence, we
obtain the following statement as a direct consequence of Theorem
\ref{Majorations of TC}:
\begin{corollary}
Let $X$ be a $(q-1)$-connected ($q\geq 1$) finite dimensional
CW-complex. If $2\dim(X)\leq 2q-2+q\cat(C_{\Delta_X}),$ then
$\TC(X)\leq \cat(C_{\Delta_X})$.
\end{corollary}

Now we turn to the inequality $\cat(C_{\Delta _X})\leq \TC(X)$.
First we note that in \cite{A-S} M. Arkowitz and J. Strom proved
that if $A\stackrel{h}{\longrightarrow
}Y\stackrel{\alpha}{\longrightarrow }C_h$ is a cofibre sequence
and $f:X\rightarrow Y$ any map, then $\secat{(\alpha f)}\leq
\secat (f)+1;$ in particular $\cat (C_h)=\secat (\alpha h)\leq
\secat (h)+1.$ Specializing this result to $X\stackrel{\Delta
_X}{\longrightarrow }X\times X\stackrel{\alpha}{\longrightarrow
}C_{\Delta _X}$ we obtain $$\cat (C_{\Delta _X})\leq \TC (X)+1$$
\noindent without any assumptions on the space $X$. This is
essentially equivalent to Lemma 18.3 of \cite{Far3}. By looking at
this inequality in terms of open covers and local sections, it is
easy to see that, if for $\TC(X)\leq n$ we have $n+1$ local
sections of $\pi:X^I\to X\times X$ satisfying the additional
condition that $s(x,x)$ is the constant path in $x$ (as long as
$s(x,x)$ has sense), then we can construct a cover of
$C_{\Delta_X}$ by $n+1$ open sets, each of which is contractible
in $C_{\Delta_X}$. In others words, under this additional
condition, we have $TC(X)\geq \cat (C_{\Delta_X})$. Another (very
close) extra condition is that the section $s$ of
$j^n_{\pi}:\ast^n_{X\times X}X^{I}\to X\times X$ required for
$\TC(X)\leq n$ satisfies $s\Delta_X\simeq \delta$ where $\delta:
X\to \ast^n_{X\times X}X^{I}$ is the canonical map induced by the
map $X\to X^I$ which sends $x$ to the constant path in $x$
(compare to the notion of ${\rm relcat}(\Delta)$ introduced in
\cite{DoeraeneElHaouari}).

As we will see, such an extra condition can be
performed when $X$ is an $H$-space but, in general, we need a
condition on the dimension and connectivity of $X$ in order to
establish $\TC(X)\geq \cat(C_{\Delta_X})$. In contrast it is easy
to see $\TC(X)\geq
\sigma^1 cat(C_{\Delta_X})$ without any extra condition. In
summary, we shall prove the following theorem:

\begin{theorem}
\label{cofibrediag<=TC}{\rm Let $X$ be a $(q-1)$-connected
CW-complex ($q\geq 1$). The following statements hold:
\begin{enumerate}
\item[(1)] $\sigma^1 cat(C_{\Delta _X})\leq\TC(X)$.
\item[(2)] If $\dim(X)\leq q(\TC(X)+1)-2,$ then
$\cat(C_{\Delta _X})\leq\TC(X)$.
\item[(3)] If $X$ is an $H$-space, then $\cat(C_{\Delta _X})\leq\TC(X)$.
\end{enumerate}
}
\end{theorem}

\begin{proof} Consider the following pushout diagram
defining $C_{\Delta _X}$
$$\xymatrix{
X \ar[r]^(.4){\Delta _X} \ar@{ >->}[d] & X\times X \ar[d]^{\alpha} \\
CX \ar[r]_(.4){\chi }& C_{\Delta _X} }$$
and the following commutative diagram
$$\xymatrix{
 X^{I} \ar[r]^{\xi}
\ar@{->>}[d]^{\pi} & PC_{\Delta _X} \ar@{->>}[d]^{ev_1}\\
 X\times X \ar[r]_{\alpha} & C_{\Delta
_X} }$$ where the vertical maps are evaluation fibrations and the
map $\xi$ comes from the homotopy triviality of $\alpha \pi$.
Taking the $n$-th fibre join of the two vertical fibrations we
obtain the following diagram in which $j^n_{ev_1}$ is equivalent
to the $n$-th Ganea fibration of $C_{\Delta_X}$:
$$\xymatrix{
\ast^n_{X\times X}X^{I} \ar[r]^{} \ar@{->>}[d]^{j^n_{\pi}} &
\ast^n_{C_{\Delta _X}}PC_{\Delta _X} \ar@{->>}[d]^{j^n_{ev_1}}\\
 X\times X \ar[r]_{\alpha} & C_{\Delta
_X} }$$ This diagram permits us to prove (1). Indeed, since
$\Delta_X$ admits a retraction, the cofibre sequence
$$X\stackrel{\Delta _X}{\longrightarrow }X\times
X\stackrel{\alpha}{\longrightarrow }C_{\Delta _X}$$ splits after
suspension. Thus $\Sigma \alpha$ admits a homotopy section, and a
homotopy section of $j^n_{\pi}$ will induce, after suspension, a
homotopy section of $\Sigma j^n_{ev_1}.$ As $\Sigma j^n_{ev_1}$
is equivalent to the suspension of the $n$-th Ganea fibration we
get $\sigma^1cat(C_{\Delta _X})\leq \TC(X).$

Now we pull back the fibrations $j^n_{\pi}$ and $j^n_{ev_1}$ along
$\Delta_X$ and $\chi$ respectively. As recalled at the end of Section \ref{joins}, this operation is equivalent to pulling back
the fibrations ${\pi}$ and ${ev_1}$ along $\Delta_X$ and $\chi$
and then taking the $n$-join of the induced fibrations. Then we
can obtain the following commutative cube, where $F$ is
homotopically equivalent to $\Omega C_{\Delta_X}$ and $\lambda$ is
induced by $\xi$ by the pullback property:
%\ar'[r][rr]^
$$\xymatrix{
\ast^n_XX^{S^1} \ar[rr]
\ar@{.>}[rd]_{\lambda}\ar@{->>}[dd]^{j^n_{\bar{\pi}}}
& & \ast^n_{X\times X}X^{I} \ar[rd]^{} \ar@{->>}'[d][dd]^(.4){j^n_{\pi}} \\
& \ast^n_{CX}F\ar[dd] \ar[rr] & & \ast^n_{C_{\Delta _X}}PC_{\Delta _X} \ar@{->>}[dd]^{j^n_{ev_1}} \\
X \ar'[r][rr]_(.4){\Delta _X} \ar[rd] && X\times X \ar[rd]_{\alpha} \\
& CX \ar[rr]_{\chi} & & C_{\Delta _X} }$$ Suppose that $\TC(X)\leq
n$ and let $s$ be a section of $j^n_{\pi}$. By the pullback along
$\Delta _X$, $s$ induces a section $\bar s$ of $j^n_{\bar{\pi}}$.
First we prove that, if the map $\lambda \bar s$ is homotopically
trivial, then $\cat(C_{\Delta _X})\leq n.$ Later we will see that
the hypothesis given in the statements of (2) and (3) items permit to
assume this fact. Under the hypothesis $\lambda \bar s\simeq *$, it
is possible to factorize $\lambda \bar s$ through the cone $CX$
giving rise to a map $\tilde s: CX \to \ast^n_{CX}F$.
%the following situation:
%$$\xymatrix{
%X \ar[rr]^{\Delta _X} \ar[rd] \ar[dd]_{\bar s} & &
%X\times X \ar[rd]^{\alpha} \ar'[d][dd]_(.3){s} \\
%& CX \ar[rr]_(.3){\chi}\ar[dd]_(.3){\tilde s} & & {C_{\Delta _X}} \ar@{.>}[dd]^{\sigma } \\
%\ast^n_XX^{S^1} \ar'[r][rr] \ar[rd]_{\lambda}
%\ar@{->>}[dd]^{j^n_{\bar{\pi}}}
%& & \ast^n_{X\times X}X^{I} \ar[rd]^{} \ar@{->>}'[d][dd]^(.4){j^n_{\pi}} \\
%& \ast^n_{CX}F\ar[dd] \ar[rr] & & \ast^n_{C_{\Delta _X}}PC_{\Delta _X} \ar@{->>}[dd]^{j^n_{ev_1}} \\
%X \ar'[r][rr]_(.4){\Delta _X} \ar[rd] & & X\times X \ar[rd]_{\alpha} \\
%& CX \ar[rr]_{\chi}&& C_{\Delta _X} }$$
By the push-out property, the maps $s$, $\bar s$ and
$\tilde s$ induce a map $\sigma: C_{\Delta _X} \to
\ast^n_{C_{\Delta _X}}PC_{\Delta _X},$ and the Gluing Lemma
\cite[II.1.2]{B} implies that the composite $j^n_{ev_1} \sigma$ is
a homotopy equivalence. The composition of $\sigma $ with an
inverse of this homotopy equivalence produces a homotopy section
of $j^n_{ev_1},$ proving that $\cat(C_{\Delta _X})\leq n.$

It remains to see that the hypothesis given in the statements of items (2)
and (3) ensure that $\lambda \bar s$ is homotopically trivial.
For (2), since $X$ is $(q-1)$-connected we have that
$\ast^n_{CX}F\simeq *^n\Omega C_{\Delta _X}$ is $q(n+1)-2$
connected. Therefore, the hypothesis on $\mbox{dim}(X)$ implies
immediately that the map $\lambda \bar s$ is homotopically
trivial.

Before considering the last case (when $X$ is an $H$-space) we
note that in general the canonical section $s_0$ of the fibration
$j^n_{\bar{\pi}}$, that is the section induced by the canonical
section of $\bar{\pi}$ that sends $x$ to the constant loop at $x$,
satisfies $\lambda s_0\simeq *$. It suffices to check that
property at the first stage, for $n=0$, and this is easily done
through an explicit description of the map $\xi$. Notice also that $s_0$ followed by the map $*^n_X X^{S^1}\to *^n_{X\times X}X^I$ is precisely the canonical map $\delta: X\to *^n_{X\times X}X^I$.

Now, when $X$ is
an $H$-space,  we consider, as in the proof of Corollary
\ref{TC(X)<=cat(Y)-group}, a right inverse $j:X\to X$ of the
identity and the map $g:X\times X\to X$ given by
$g(x,y)=\mu(x,j(y))$ where $\mu$ is the multiplication. Since
$g\Delta _X$ is homotopically trivial we have the two following
commutative diagrams:
$$\begin{array}{cc}
\xymatrix{
 X \ar[r]^{\Delta _X}
\ar[d]^{} & X\times X \ar[d]^{g}\\
 CX \ar[r]_{H} & X }
 &
\xymatrix{
 X^{I} \ar[r]^{\bar g}
\ar@{->>}[d]^{\pi} & PX \ar@{->>}[d]^{ev_1}\\
 X\times X \ar[r]_{g} & X }
 \end{array}$$
where $H$ is the null-homotopy and $\bar g$ is given, for $\beta
\in X^I$, by
$$\bar g(\beta)(t)=\left\{\begin{array}{lr}
H([\beta(1),1-2t]) & 0\leq t\leq \frac{1}{2}\\
\mu(\beta(2-2t),j(\beta(1)) & \frac{1}{2}\leq t\leq 1
\end{array}\right.$$
Here $[x,s]$ stands for the class of $(x,s)$ in $CX=X\times
I/X\times 1$. Using the fact that $\bar g$ induces a homotopy
equivalence between the fibres and that $g=p_1\circ \theta$ where
$p_1:X\times X \to X$ is the projection on the first factor and
$\theta:X\times X\to X\times X$ is the homotopy equivalence given
by $\theta(x,y)=(g(x,y),y)$
 we can check that the right-hand diagram is a homotopy pullback.
We denote by $W$ the pullback of the fibration $ev_1:PX\to X$
along the homotopy $H$. Then we have the following commutative
cube in which $\gamma$ is induced by $\bar g$ and the vertical
faces are homotopy pullbacks:
$$\xymatrix{
\ast^n_XX^{S^1} \ar[rr]
\ar@{.>}[rd]_{\gamma}\ar@{->>}[dd]^{j^n_{\bar{\pi}}}
& & \ast^n_{X\times X}X^{I} \ar[rd]^{} \ar@{->>}'[d][dd]^(.4){j^n_{\pi}} \\
& \ast^n_{CX}W\ar@{->>}[dd] \ar[rr] & & \ast^n_{X}PX \ar@{->>}[dd]^{j^n_{ev_1}} \\
X \ar'[r][rr]_(.4){\Delta _X} \ar[rd] && X\times X \ar[rd]_{g} \\
& CX \ar[rr]_{H} & & X }$$ Since $CX$ is contractible the
fibration $j^n_{\bar{\pi}}$ is homotopically equivalent to the
trivial one $X\times \ast^n_{CX} W \to X$ and $\ast^n_{CX} W
\simeq \ast^n \Omega X$. The canonical section $s_0$ of
$j^n_{\bar{\pi}}$ satisfies $\gamma s_0\simeq *$ (it suffices to
check it at the first level $n=0$) and is actually the unique
section, up to homotopy, having this property. Now we come back to
our sections $s$ and $\bar s$. Since $X$ is an $H$-space we know
that $\TC(X)=\cat(X)$ and thus we can suppose that $s$ is induced
by a section of $j^n_{ev_1}: \ast^n_{X}PX\to X$. Using the
commutative cube above we see that this fact implies that $\bar s$
is induced by a section of the fibration $\ast^n_{CX} W \to CX$.
This shows in particular that $\gamma \bar s$ factors through the
cone $CX$ and therefore is homotopically trivial. As a consequence
$\bar s$ must be homotopic to the canonical section $s_0$ and
$\lambda \bar s\simeq \lambda s_0$ is homotopically trivial. This achieves the proof of (3).
\end{proof}

\noindent \textbf{Remark.} When $X$ is a topological group (with
unit element as the base point) we can give a more direct argument
of point (3). Indeed, we can consider $g(x,y)=x\cdot y^{-1}$.
Assuming that $\TC(X)\leq n$ we know that $\cat(X)\leq n$. If
$\zeta:V\to PX$ is a local pointed section of $ev_1:PX\to X,$ then
we obtain on $U=g^{-1}(V)$ a local section $\zeta '$ of $\pi:X^
I\to X\times X$ given by $\zeta '(x,y)(t)=\zeta(x\cdot
y^{-1})(1-t)\cdot y$. This local section sends $(x,x)$ to the
constant path in $x$. In this manner we can obtain a cover of $X\times X$
by $n+1$ open sets together with local sections which send a point $(x,x)$
to the constant path in $x$. Therefore, as mentioned before the statement
of Theorem \ref{cofibrediag<=TC}, this cover will induce a cover of $C_{\Delta_X}$ by $n+1$ open sets, each of
which is contractible in $C_{\Delta_X}$. Hence
$\cat(C_{\Delta_X})\leq n$. Note that this argument uses the
associativity of the multiplication while the argument given in
the proof above does not require any associativity condition.

\subsection{First examples of the equality
$\TC(X)=\cat(C_{\Delta_X})$}\label{examples TC=catcofibre}

\begin{enumerate}
\item[(i)] The first example is quite trivial.
If $X$ is contractible, then $\TC(X)=\cat(C_{\Delta_X})$. It
suffices to note that in this case $C_{\Delta_X}$ is also
contractible.

\item[(ii)] For any $n\geq 1$, $\TC(S^n)=\cat(C_{\Delta_{S^n}})$.

As Farber showed in \cite{Far}, $\TC(S^n)$ is $1$ if $n$ is odd
and $2$ if $n$ is even. In \cite{Weaksecat} we established that
$C_{\Delta_{S^n}}$ is homotopy equivalent to
$S^n\cup_{[\iota_n,\iota_n]}e^{2n}$ where $\iota_n$ denotes the
homotopy class of the identity of $S^n$ and $[\iota_n,\iota_n]$
the Whitehead product. Using the results of \cite{B-H} together
with the classical results on the Hopf invariant of the Whitehead
product $[\iota_n,\iota_n]$, it is easy to see that the category
of this space is also $1$ if $n$ is odd and $2$ if $n$ is even.
Thus we have the equality $\TC(S^n)=\cat(C_{\Delta_{S^n}})$. Note
that Theorem \ref{Majorations of TC} together with the knowledge of
$\cat(C_{\Delta_{S^n}})$ and of $\nil\ker \cup $ for the spheres
permit us to recover the value of $\TC $ for theses spaces.

\item[(iii)]\label{TC=CatCofibre-group}If $X$ is a $H$-space that
is a connected CW-complex, then $ \TC(X)=\cat(C_{\Delta_X})$.

As we have previously done, we consider a right inverse $j:X\to X$
of the identity and the map $g:X\times X\to X$ given by
$g(x,y)=\mu(x,j(y))$ where $\mu$ is the multiplication. Since $g
\Delta_X$ is trivial $g$ factors, up to homotopy, through
$C_{\Delta_X}$:
$$\xymatrix{
X\ar[r]^{i_1}&X\times X \ar[d]_{\alpha}\ar[r]^g & X\\
& C_{\Delta_X} \ar[ru]_{\tilde g} }$$ Since $g i_1\simeq id_X$ we
obtain $\tilde g \alpha  i_1\simeq id_X$. In other words, $X$ is a
homotopy retract of $C_{\Delta_X}$ and therefore $\cat(X)\leq
\cat(C_{\Delta_X}).$ Since $\cat(X)=\TC(X)$ by Corollary
\ref{TC(X)<=cat(Y)-group}, we get the inequality $\TC(X)\leq
\cat(C_{\Delta_X})$. The other inequality follows from Theorem
\ref{cofibrediag<=TC}.

%Using this result we can see that the equality $\TC(X)=
%\cat(C_{\Delta_X})$ holds in particular for $X=\RP^3, \RP^7$ and
%$Sp(2).$

\item[(iv)]
 Let $\CP^n$ denote the $n$-th complex
projective space. For any $n\geq 1$ we have that $\TC(\CP^n)=
\cat(C_{\Delta_{\CP^n}})$. Indeed, by Farber (\cite{Far}) we know
that
$$\TC(\CP^n)=\nil\ker \cup=2n.$$ In \cite{Weaksecat} we established
that, for any space $X$, $\nil\ker \cup\leq
\wTC(X)=\wcat(C_{\Delta_X})$. Thus we obtain
$\wcat(C_{\Delta_{\CP^n}})\geq 2n$ and therefore
$\cat(C_{\Delta_{\CP^n}})\geq 2n$. Since $\CP^n$ is $1$-connected
and $2n$-dimensional we have, by Theorem \ref{cofibrediag<=TC},
$\cat(C_{\Delta_{\CP^n}})\leq \TC(\CP^n)=2n$. In conclusion,
$\cat(C_{\Delta_{\CP^n}})=\TC(\CP^n)=2n$.

\item[(v)]
In general, $\TC(X)=\cat(C_{\Delta_{X}})$ for $X$ any closed,
$2n$-di\-men\-sional, 1-connected symplectic manifold. Again
$\TC(X)=\nil\ker \cup=2n$ (\cite{F-T-Y}) and by Theorem
\ref{cofibrediag<=TC}, $\cat(C_{\Delta_{X}})=\TC(X)=2n$.

\end{enumerate}
The argument in the previous two examples can be summarized as
follows: if $X$ is a $(q-1)$-connected CW-complex such that
$\TC(X)=\nil\ker \cup$ and $\dim(X) \leq q(\TC(X)+1)-2,$ then one
has $\TC(X)= \cat(C_{\Delta_{X}})$. Using this observation, we
obtain the following results:

\begin{enumerate}

\item[(vi)] For any compact orientable surface of genus $g$,
$X=\Sigma_g$, one has $\TC(X)= \cat(C_{\Delta_{X}})$.

Here we use the result of \cite{Far} that,
$\TC(\Sigma_g)=\nil\ker\cup=\left\{\begin{array}{cc} 2 & \mbox{ if } g\leq 1\\
4 & \mbox{ if } g> 1.\end{array}\right.$ Notice that the condition
$\dim(X) \leq q(\TC(X)+1)-2$ is not satisfied when $g=1$ but, in
this case, the result follows from Example (iii) since
$\Sigma_1=S^1\times S^1$ is an $H$-space.

\item[(vii)] For any connected finite graph $X$, one has $\TC(X)= \cat(C_{\Delta_{X}})$.

Here we use the result of \cite{Far2} that
$\TC(X)=\nil\ker\cup=\left\{\begin{array}{cc} 0 & \mbox{ if } b_1(X)=0\\
1 & \mbox{ if } b_1(X)=1\\ 2 & \mbox{ if }
b_1(X)>1,\end{array}\right.$ where $b_1(X)$ is the first Betti
number. Notice that the condition $\dim(X) \leq q(\TC(X)+1)-2$ is
not satisfied in the first two cases but, in these cases, the
result follows from Examples (i) and (ii) since $X$ is either
contractible or homotopy equivalent to $S^1$.

\item[(viii)] Let $X=F(\R^m,n)$ be the space of configurations of $n$ distinct
points in $\R^m$ with $m\geq 2$ and $n\geq 2$. One has $\TC(X)=
\cat(C_{\Delta_{X}})$.

Here we use the fact that $X$ is an $(m-2)$-connected space, which
is homotopy equivalent to a finite CW-complex of dimension less
than or equal to the product $(m-1)(n-1),$ and the result by
Farber-Yuzvinsky \cite{F-Y} and Farber-Grant \cite{Farber-Grant}
that
$$\TC(X)=\nil\ker\cup=\left\{\begin{array}{cc} 2n-2 & \mbox{ if }m \mbox{ odd} \\
2n-3 & \mbox{ if }m \mbox{ even.}\end{array}\right.$$
%The only case in which the condition $\dim(X) \leq q(\TC(X)+1)-2$ is not
%satisfied is
The only case in which the condition $\dim(X) \leq q(\TC(X)+1)-2$
is not satisfied is when $n=m=2$ but, in this case
$X=F(\R^2,2)\simeq S^1$ and the result follows from  Example (ii)
(or (iii)). We observe that, in general, $F(\R^2,n)$ is homotopy
equivalent to $X\times S^1,$ where $X$ is a finite polyhedron of
dimension less than or equal to $n-2$ (\cite{F-Y}).
\end{enumerate}

\section{The case of real projective spaces and standard lens spaces.}
\label{section RP and lens spaces}

Recall that the $n$ dimensional real projective space is
$\RP^n=S^n/\mathbb{Z}_2$ and that the (standard) $2n+1$
dimensional $p$-torsion lens space is given by $L_p^ {2n+1}=S^
{2n+1}/\mathbb{Z}_p$. Here $p>2$ is an odd integer and the action
of $\mathbb{Z}_p,$ identified with the multiplicative group
$\{1,\omega,...,\omega^{p-1}\}\subset \C$ of $p$-th roots of
unity, on $S^{2n+1}\subset \C^{n+1}$ is given by the pointwise
multiplication. The infinite real projective space and the
infinite $p$-torsion lens space are respectively
$\RP^{\infty}=S^{\infty}/\mathbb{Z}_2$ and $L^{\infty
}_p=S^{\infty}/\mathbb{Z}_p$.

The aim of this section is to establish the following result:

\begin{theorem}\label{Theorem RPn-lens spaces}{\rm Let $p>2$ be an odd integer and $n\geq 0$ any integer.
\begin{enumerate}
\item[(a)] For any $n\geq 1$ we have the equality $\TC(\RP^n)=
\cat(C_{\Delta_{\RP^n}}).$
\item[(b)] If $n+1$ and $p$ are not powers of a common odd prime, then
$\mbox{TC}(L^{2n+1}_p)=\mbox{cat}(C_{\Delta _{L^{2n+1}_p}}).$
\end{enumerate}}
\end{theorem}

As mentioned in the introduction, a consequence of this result is the following reformulation of Theorem \ref{TC-immersion}:

\begin{corollary}{\rm
For $n\neq 1,3,7$ the least integer $k$ such that
$\RP^n$ can be immersed in $\R^k$ is equal to $\cat(C_{\Delta_{\RP^n}}).$}
\end{corollary}

The proof of Theorem \ref{Theorem RPn-lens spaces} is based on
Theorems \ref{Majorations of TC} and \ref{cofibrediag<=TC} but
requires more material which will be developed in the following
two sections. The proof of Theorem \ref{Theorem RPn-lens spaces}
is given in Section \ref{Proof-Theorem RPn-lens spaces}.

\subsection{Real projective spaces, lens spaces and Ganea fibrations}

One ingredient of the proof of Theorem \ref{Theorem RPn-lens
spaces} is the relation, given in the following lemma, between the
inclusion $\RP^n\hookrightarrow \RP^{\infty}$ (resp.
$L_p^{2n+1}\hookrightarrow L_p^{\infty}$ or, more generally,
$(L_p^{\infty})^{(k)}\hookrightarrow L_p^{\infty}$) and the $nth$
Ganea fibration of $\RP^{\infty}$ (resp. $L_p^{\infty}$). Here
$(L_p^{\infty})^{(k)}$ denotes the $k$ skeleton of $L_p^{\infty}$
with respect to the standard CW-decomposition of $L_p^{\infty}$
(see \cite{Whitehead}).

\begin{lemma}\label{ganea}{\rm
For each $k\geq 0$ there exist homotopy commutative diagrams of the following form:
$$\xymatrix{
{\RP^k} \ar[r] \ar@{^(->}[rd] & {G_k(\RP^{\infty})} \ar[r]\ar[d]^{g_k} & {\RP^k} \ar@{^(->}[ld] \\
 & {\RP^{\infty}} & }
\quad
\xymatrix{
{(L^{\infty }_p)^{(k)}} \ar[r] \ar@{^(->}[rd] & G_k(L^{\infty }_p)
\ar[r] \ar@{>>}[d]^{g_k} & {(L^{\infty }_p)^{(k)}}\ar@{^(->}[ld]\\
&L^{\infty }_p&
}$$
 }
\end{lemma}

In order to prepare the proof of this lemma we recall that, since
$S^{\infty}$ is contractible, $\RP^{\infty}$ and $L^{\infty }_p$
can be identified with the Milnor classifying spaces $B{\mathbb
Z}_2$ and $B{\mathbb Z}_p,$ respectively. On the other hand, when
$X$ is the classifying space $BG$ of a group $G$, the $n$-th Ganea
fibration of $X=BG$ and the inclusion $B_nG\hookrightarrow BG$ of
the $n$-th stage of the classifying construction are known to be
related. In particular, for each $n\geq 0$, there exists a
homotopy commutative diagram of the following form:

%\footnote{Jose, I put away the hypothesis that $EG\to BG$ is
%equivalent, as a principal $G$-bundle, to a fibration because,
%according to Selick (Chapter IX), it follows from May that a
%numerable fibre bundle is a fibration and, on the other hand, when
%$G$ is a group, $EG\to BG$ is a numerable fibre bundle}

\begin{equation}\label{Ganea-Milnor construction}
\xymatrix{
B_nG \ar[r] \ar@{^(->}[rd] & G_n(BG) \ar[r]\ar[d]^{g_n} & B_nG\ar@{^(->}[ld]\\
&BG& }
\end{equation}
The left-hand side follows from the fact that $\cat(B_nG)\leq n$
(see, for instance, \cite{Selick}) and the right-hand side can be
established inductively through the fibre-cofibre process.

When $G=\mathbb{Z}_2$, it is well-known that the inclusion
$B_n\mathbb{Z}_2\hookrightarrow B\mathbb{Z}_2$ coincides with the
inclusion $\RP^n\hookrightarrow \RP^{\infty}$ so that the first
part of Lemma \ref{ganea} is established. Actually we point out
that, in the real projective spaces case, the inclusion
$\RP^n\hookrightarrow \RP^{\infty}$ is, up to homotopy
equivalence, the Ganea fibration $G_n(\RP^{\infty})\to
\RP^{\infty}$. This follows from the fact that
$\RP^n\hookrightarrow \RP^{\infty}$ can be obtained from
$\RP^{n-1}\hookrightarrow \RP^{\infty}$ through the fibre-cofibre
construction.

By contrast, when $G=\mathbb{Z}_p$, the inclusion
$B_n\mathbb{Z}_p\hookrightarrow B\mathbb{Z}_p$ does not coincide
with the inclusion $(L_p^{\infty})^{(n)}\hookrightarrow
L_p^{\infty}$ so that we will need to be more explicit in this
case.

Recall that the $G$-principal bundle $EG\to BG$ can be constructed
as follows: consider the cone $CG=G\times I/G\times 0$ and, for
each $n\geq 0$, consider the subspace $E_nG\subset (CG)^{n+1}$
consisting of the elements $((g_0,t_0),...,(g_n,t_n))$ such that
$\sum t_i=1$. Thus $E_nG\subset E_{n+1}G$ (through the
identification
$((g_0,t_0),...,(g_n,t_n))=((e,0),(g_0,t_0),...,(g_n,t_n))$ where
$e$ is the unit element and the base point of $G$) and $EG$ is
defined as $\bigcup\limits_nE_nG$ equipped with the weak topology.
The elements of $E_nG$ (resp. $EG$) are written as
$\sum\limits_{i=0}^ng_it_i$ (resp. $\sum\limits_{i\geq 0}
g_it_i$). The spaces $B_nG$, $BG$ are defined as the orbits spaces
$E_nG/G$, $EG/G$ with respect to the action $g\cdot \sum g_it_i=
\sum g\cdot g_it_i.$
%There is, for each $n$, a commutative
%diagram of $G$-principal bundles
%$$\xymatrix{
%E_nG \ar[d] \ar[r] & EG \ar[d]\\
%B_nG \ar[r] & BG }$$ which is a pullback.

\noindent For $G=\mathbb{Z}_p$, $\omega$ a primitive $p$-th root of $1$ and $n\geq 0$, the map
$$\begin{array}{rcl}
E_{2n+1}\mathbb{Z}_p & \to & S^{2n+1}\subset \C^{n+1}\\
\sum\limits_{i=0}^{2n+1}g_it_i& \mapsto &
\disfrac{(t_0\omega^{g_0}+t_1\omega^{g_1},...,t_{2n}\omega^{g_{2n}}+t_{2n+1}\omega^{g_{2n+1}})}
{||(t_0\omega^{g_0}+t_1\omega^{g_1},...,t_{2n}\omega^{g_{2n}}+t_{2n+1}\omega^{g_{2n+1}})||}
\end{array}$$
is well-defined (since $p$ is odd, the vector
$(t_0\omega^{g_0}+t_1\omega^{g_1},...,t_{2n}\omega^{g_{2n}}+t_{2n+1}\omega^{g_{2n+1}})$
does not vanish) and is $\mathbb{Z}_p$-equivariant. Thus we obtain
a map $B_{2n+1}\mathbb{Z}_p \to L_p^{2n+1}$ which is compatible
with the inclusions and which is, at the limit, a homotopy
equivalence (since $E\mathbb{Z}_p$ and $S^{\infty}$ are both
contractible):
\begin{equation}
\label{Milnor-lens spaces}
\xymatrix{
B_{2n+1}\mathbb{Z}_p \ar[d]\ar[r] & L_p^{2n+1} \ar[d] \\
 B\mathbb{Z}_p
\ar[r]^{\simeq } & L_p^{\infty} }
\end{equation}
%Because these diagrams are
%pullbacks, the map $E\mathbb{Z}_p\to B\mathbb{Z}_p$ is a
%fibration.
%The homotopy equivalences in the right diagram come
%from the fact that

\noindent Recall also that the standard CW-decomposition of
$L_p^{\infty}$ is induced by the CW-decomposition of $S^{\infty }$
consisting of the cells
$$\begin{array}{l}
E^{2k}=\{(z_0,...,z_k)\in S^{2k+1}:z_k\in \R \mbox{ and }z_k\geq 0\}\\
E^{2k+1}=\{(z_0,...,z_k)\in S^{2k+1}:0\leq \mbox{arg}(z_k)\leq \frac{2\pi }{p}\}
\end{array}$$
where $k\geq 0$ and by their images under the action of
$\mathbb{Z}_p$. In this decomposition the cells of dimension $\leq
2n+1$ give a cell decomposition of $S^{2n+1}.$ The images of the
cells $E^k$ under the covering map $q:S^{\infty }\rightarrow
L^{\infty}_p$ give a CW-decomposition of $L^{\infty}_p$ with just
one cell on each dimension and such that
$(L^{\infty}_p)^{(2n+1)}=L^{2n+1}_p.$ Since, in this
CW-decomposition of $S^{\infty}$, the $(2n+2)$-skeleton can be
identified to a cone over the $(2n+1)$-skeleton (which is
$S^{2n+1}$), we can identify the covering map $S^{2n+1}\to
L^{2n+1}_p$ with the attaching map of the $2n+2$ dimensional cell
in $L_p^{\infty}$. In other words, the following commutative
diagram is a homotopy pushout:
\begin{equation}\label{attachingmapevencells}
\xymatrix{
{S^{2n+1}} \ar@{^{(}->}[r] \ar[d]_{} &
{(S^{\infty})^{(2n+2)}} \ar[d]^q  \\
{L }_p^{2n+1} \ar@{^{(}->}[r] & {(L^{\infty }_p)^{(2n+2)}}
 }
 \end{equation}
We also observe that the covering maps $S^{2n+1}\to L^{2n+1}_p$
and $S^{\infty}\to L_p^{\infty}$ are fibrations and fit in a
commutative diagram which is a pullback and a homotopy pullback:
\begin{equation}
\label{homotopyfibreoftheinclusion}\xymatrix{
{S^{2n+1}} \ar@{^{(}->}[r] \ar@{>>}[d]_{} &
{S^{\infty }} \ar@{>>}[d]^q  \\
{L_p^{2n+1}} \ar@{^{(}->}[r] & {L^{\infty }_p} }
\end{equation}

Now we come to the proof of Lemma \ref{ganea}.

\begin{proof} As said above, the real projective spaces case follows directly from diagram (\ref{Ganea-Milnor construction}).
So, we just have to prove the existence, for any $k\geq 0$, of a homotopy commutative diagram
$$\xymatrix{
{(L^{\infty }_p)^{(k)}} \ar[r] \ar@{^(->}[rd] & G_k(L^{\infty }_p)
\ar[r] \ar@{>>}[d]^{g_k} & {(L^{\infty }_p)^{(k)}}\ar@{^(->}[ld]\\
&L^{\infty }_p& }$$ The left part comes from the fact that
$\cat({(L^{\infty }_p)^{(k)}})\leq \dim({(L^{\infty
}_p)^{(k)}})=k$. For $k=0,$ we also have the right part since
$(L^{\infty}_p)^{(0)}=*$. Suppose now that $k=2n+1$ with $n\geq
0$. Using diagrams (\ref{Ganea-Milnor construction}),
(\ref{Milnor-lens spaces}) and the fact that $L_p^{\infty}$ and
$B\mathbb{Z}_p$ are homotopy equivalent, we easily see that, for
each $n\geq 0$, there exists a homotopy commutative diagram of the
following form:
$$\xymatrix{
G_{2n+1}(L_p^{\infty}) \ar[dr]_{g_{2n+1}} \ar[rr]&& L_p^{2n+1}=(L_p^{\infty})^{(2n+1)} \ar[dl]\\
&L_p^{\infty}  & }$$ which establishes the lemma for $k=2n+1$. We
now apply the fibre-cofibre process to this diagram. On the left
side we will obtain the $(2n+2)$-th Ganea fibration of
$L_p^{\infty}$. On the other hand, from diagrams
(\ref{homotopyfibreoftheinclusion}) and
(\ref{attachingmapevencells}), we know that the homotopy fiber of
the inclusion $L_p^{2n+1}\to L_p^{\infty}$ is $S^{2n+1},$ that the
induced map $S^{2n+1}\to L_p^{2n+1}$ is the covering map and that
the homotopy cofibre of this map is equivalent to
$(L_p^{\infty})^{(2n+2)}$. Then, through the fibre-cofibre
process, the previous diagram induces a homotopy commutative
diagram of the following form
$$\xymatrix{
G_{2n+2}(L_p^{\infty}) \ar[dr]_{g_{2n+2}} \ar[rr]&& (L_p^{\infty})^{(2n+2)} \ar[dl]\\
&L_p^{\infty}  &
}$$
which completes the proof.
\end{proof}

Just as a remark, we point out that we have used in the previous proof the inequality $\cat({(L^{\infty
}_p)^{(k)}})\leq k$. Actually we have:

\begin{lemma}\label{cat-lens}
{\rm$\cat({(L^{\infty }_p)^{(k)}})=k,$ for all $k\geq 0.$ }
\end{lemma}
\begin{proof}
It is a well-known fact that $\mbox{cat}(L^{2n+1}_p)=2n+1,$ for
all $n\geq 0$ (see for instance \cite{C-L-O-T}). Since $2n+1=\cat
(L^{2n+1}_p)\leq \cat ((L^{\infty }_p)^{(2n)})+1$ we have that
$\cat ((L^{\infty }_p)^{(2n)})\geq 2n$ and therefore $\cat ((
L^{\infty }_p)^{(2n)})=2n.$
\end{proof}

\subsection{Axial maps}

For our main result we shall also deal with
\emph{$\mathbb{Z}_2$-axial maps} in the case of real projective
spaces and \emph{$\mathbb{Z}_p$-axial maps} in the case of lens
spaces. We begin with the notion of $\mathbb{Z}_2$-axial maps
which is the analogous of the notion of \emph{$\mathbb{Z}_p$-axial
maps} considered in \cite{G} and is slightly more general than the
classical notion of axial map:

\begin{definition}{\rm A map $a:\RP^n\times \RP^n\rightarrow \RP^k$ is said to be
\emph{$\mathbb{Z}_2$-axial} when the composite

$${\RP^n\times \RP^n} \stackrel{a}{\longrightarrow }{\RP^k}{\hookrightarrow
}{\RP^{\infty }}$$ \noindent classifies the $\mathbb{Z}_2$-cover
$S^n\times _{\mathbb{Z}_2 }S^n\rightarrow \RP^n\times \RP^n.$ In
other words, there is a homotopy commutative diagram of the form
$$\xymatrix{
{\RP^n\times \RP^n} \ar[rr]^a \ar@{^{(}->}[d] & &
{\RP^k} \ar@{^{(}->}[d]^{j_k} \\
{\RP^{\infty }\times \RP^{\infty }} \ar[rr]_{m} & & {\RP^{\infty
}} }$$ \noindent where $m$ stands for the multiplication coming
from the $H$-group structure of $\RP^{\infty }
=K(\mathbb{Z}_2,1)=\Omega K(\mathbb{Z}_2,2).$}
\end{definition}

By easy integral homological arguments we have that, necessarily,
$k\geq n$. For $k>n$ a $\mathbb{Z}_2$-axial map is nothing else
but an axial map of type $(n,k)$. We are interested in the case
$k=n$. We know that $\RP^n$ has an $H$-space structure for
$n=1,3,7,$ so there exists a $\mathbb{Z}_2$-axial map $\RP^n\times
\RP^n\rightarrow \RP^n$ in these cases. The key point for Theorem
\ref{Theorem RPn-lens spaces} is the following result of
non-existence which as other similar results (see, for instance,
\cite[Lemma 9]{F-T-Y}) is based on Adams's Hopf invariant one
Theorem \cite{A}.

\begin{lemma}\label{salvation}{\rm
There are not $\mathbb{Z}_2$-axial maps $a:\RP^n\times
\RP^n\rightarrow \RP^n,$ for $n\neq 1,3,7.$}
\end{lemma}

\begin{proof}
Suppose that there exists such map for some $n\neq 1,3,7.$ We
consider $H^*(\RP^n;\mathbb{Z}_2)\cong \mathbb{Z}_2[x]/(x^{n+1})$
the mod 2 cohomology ring of $\RP^n,$ being $x$ the generator of
degree 1. We also denote by $x_i$ the corresponding element in the
$i$-th factor of $\RP^n\times \RP^n .$ From $a^*(x)=x_1+x_2$ we
obtain $0=a^*(x^{n+1})=(x_1+x_2)^{n+1}$ so that $2$ divides
$\binom{n+1}{i},$ for $i=1,...,n.$ Then $n+1$ must be a power of
$2$ and therefore $n$ is odd.

Now take the following homotopy commutative diagram for each
$k=1,2:$
$$\xymatrix{
{\RP^n} \ar@{^{(}->}[d]_{j_n} \ar[rr]^{i_k} & & {\RP^n\times
\RP^n} \ar[rr]^a \ar@{^{(}->}[d] & & {\RP^n}
\ar@{^{(}->}[d]^{j_n} \\
{\RP^{\infty }} \ar[rr]_{i_k} & & {\RP^{\infty }\times \RP^{\infty
}} \ar[rr]_{m } & & {\RP^{\infty } }}$$
Since $m i_k\simeq id$ we obtain a homotopy commutative diagram
$$\xymatrix{
{\RP^n} \ar[rr]^{ai_k} \ar@{^{(}->}[drr]_{j_n} & &
{\RP^n} \ar@{^{(}->} [d]^{j_n}  \\
 & & {\RP^{\infty }} } $$
and the corresponding commutative diagram in integral $n$-homology:
$$\xymatrix{
{H_{n}(\RP^n)} \ar[r]^{(ai_k)_*} \ar[dr]_{(j_n)_*} &
{H_{n}(\RP^n)} \ar[d]^{(j_n)_*\hspace{20pt}\equiv} & {\mathbb{Z}}
\ar[r] \ar@{>>}[dr] & {\mathbb{Z}} \ar@{>>}[d] \\
 & {H_{n}(\RP^{\infty })} & &  {\mathbb{Z}_2}}$$
\noindent In particular we have that, necessarily,
$\mbox{deg}(ai_k)\equiv 1$ (mod 2), and therefore
$\mbox{deg}(ai_k)=d_k$ is an odd integer, $k=1,2.$
Since
$q:S^n\twoheadrightarrow \mathbb{R}\mbox{P}^n$ is a covering map and
$S^n\times S^n$ is simply connected, we can consider a lift
$$\xymatrix{
{S^n\times S^n} \ar@{>>}[d]_{q\times q} \ar@{.>}[rr]^{\varphi } & & {S^n} \ar@{>>}[d]^q \\
{\RP^n\times \mathbb{R}\mbox{P}^n} \ar[rr]_{a} & & {\RP^n} }$$
We thus get $\mbox{deg}(\varphi i_k)=d_k$ since we have the commutativity
$$\xymatrix{
%{S^n} \ar@{>>}[d]_{q} \ar[rr]^{\varphi i_k} & & {S^n}
%\ar@{>>}[d]^{q\hspace{35pt}\Rightarrow} & &
{H_n(S^n)} \ar[d]_{q_*} \ar[rr]^{(\varphi i_k)_*} & & {H_n(S^n)} \ar[d]^{q_*} \\
%{\RP^n} \ar[rr]_{ai_k} & & {\RP^n} & &
{H_n(\RP^n)}
\ar[rr]_{(ai_k)_*} & & {H_n(\RP^n)} }$$ \noindent and $q_*=H_n(q)$
injective. Hence, the bidegree of $\varphi $ is $(d_1,d_2).$
Observe that, in order to check that $q_*$ is injective we just
have to consider the exact homology sequence associated to the
cofibration $S^n\stackrel{q}{\longrightarrow
}\RP^n\stackrel{i}{\hookrightarrow
}\RP^{n+1}$ and take into account that
$H_{n+1}(\mathbb{R}\mbox{P}^{n+1})=0.$

Now, using the same argument as in \cite[Lemma 9]{F-T-Y}, we show
that such a map $\varphi:S^n\times S^n\to S^n $ of bidegree
$(d_1,d_2)$ where $d_1$ and $d_2 $ are both odd does not exist.
>From the existence of $\varphi $ we have indeed that, if $\iota
\in \pi _n(S^n)$ is the generator, the following Whitehead product
vanishes
$$0=[d_1\iota ,d_2\iota ]\in \pi
_{2n-1}(S^n)$$

On the other hand, since $n$ is odd with $n\neq 1,3,7,$ then
$[\iota ,\iota ]\neq 0$ and has order two by J.F. Adams \cite{A}.
Since the product $d_1d_2$ is odd, we have that $[d_1\iota ,d_2\iota
]=d_1d_2[\iota ,\iota ]\neq 0,$ which is a contradiction.
\end{proof}

In the case of lens spaces we will use the notion of $\mathbb{Z}_p$-axial maps as it appears in \cite{G}.

\begin{definition}{\rm \cite{G}
A continuous map $a:L^{2n+1}_p\times L^{2n+1}_p\rightarrow
L^{2r+1}_p$ is said to be \newline \emph{$\mathbb{Z}_p$-axial}
when the composite
$${L^{2n+1}_p\times
L^{2n+1}_p} \stackrel{a}{\longrightarrow }{L^{2r+1}_p}
\hookrightarrow {L^{\infty }_p}$$ \noindent classifies the
$\mathbb{Z}_p$-cover $S^{2n+1}\times _{\mathbb{Z}_p
}S^{2n+1}\rightarrow L^{2n+1}_p\times L^{2n+1}_p.$}
\end{definition}

\begin{remark}{\rm Similarly to the case of real projective spaces,
$L^{\infty }_p=K(\Z_p,1)=\Omega K(\Z_p,2)$ is given with an
$H$-group structure and a $\mathbb{Z}_p$-axial map should be
thought of as a deformation of the restriction to
$L^{2n+1}_p\times L^{2n+1}_p$ of the product $m(x,y)=x\cdot y$ in
$L^{\infty }_p.$ That is, there exists a homotopy commutative
diagram
$$\xymatrix{
{L^{2n+1}_p\times L^{2n+1}_p} \ar[rr]^a \ar@{^(->}[d] & &
{L^{2r+1}_p} \ar@{^(->}[d] \\
{L^{\infty }_p\times L^{\infty }_p} \ar[rr]_m & & {L^{\infty }_p}
}$$ Again, by considering easy (integral) homology arguments we
have that necessarily $n\leq r$ holds.

%Indeed, since $mi_1\simeq id,$ the diagram above induce a homotopy
%commutative diagram
%$$\xymatrix{
%{L^{2n+1 }_p} \ar[rr]^{ai_1} \ar@{^{(}->}[drr] & &
%{L^{2m+1 }_p} \ar@{^{(}->}[d]  \\
% & & {L^{\infty}_p} }$$ Passing, for instance, to the integral $(2n+1)$-th
%homology functor we obtain an impossible commutative diagram of
%groups when $n>m.$ Furthermore, the case $n=m$ neither holds.
 }
\end{remark}

The proof of the following proposition can be found in \cite[Lemma
4.1]{G}. It is quite similar to the first part of the proof in
Lemma \ref{salvation}.

\begin{proposition}\label{axial}{\rm
If $n+1$ and $p$ are not powers of a common odd prime, then there
are not $\mathbb{Z}_p$-axial maps $a:L^{2n+1}_p\times
L^{2n+1}_p\rightarrow L^{2n+1}_p.$}
\end{proposition}

In the proof of Theorem \ref{Theorem RPn-lens spaces}, we will more precisely use the following non-existence result:

\begin{corollary}\label{axial2}{\rm Let $\mu :L^{\infty
}_p\times L^{\infty }_p\to L^{\infty }_p$ given by $\mu (x,y)=x\cdot
y^{-1}$ and  $\mu': L^{2n+1}_p\times
L^{2n+1}_p\hookrightarrow L^{\infty}_p\times L^{\infty}_p
\stackrel{\mu}{\to} L^{\infty}_p$ its restriction to $L^{2n+1}_p\times
L^{2n+1}_p$.
The map $\mu'$ cannot be deformed in $L_p^{2r+1}$, in the following cases
\begin{enumerate}
\item[(i)] $n>r;$ or
\item[(ii)] $n=r,$ and $n+1$ and $p$ are not powers of a common odd
prime.
\end{enumerate}
 }
\end{corollary}

\begin{proof} As said before the argument is homological whenever $n>r$. Suppose $r=n$.
The inversion in $L^{\infty }_p=\Omega K(\Z_p,2)$ can be deformed
in a map $j:L_p^ {2n+1}\to L_p^ {2n+1}$. If there exists a map
$a:L^{2n+1}_p\times L^{2n+1}_p\to L^{2n+1}_p$ such that the
composite $L^{2n+1}_p\times L^{2n+1}_p\stackrel{a}{\to}
L^{2n+1}_p\hookrightarrow L_p^{\infty}$ is homotopic to $\mu'$
then the map
$$\xymatrix{ {L^{2n+1}_p\times L^{2n+1}_p}  \ar[r]^{id\times j} & {L^{2n+1}_p
\times L^{2n+1}_p} \ar[r]^(.6){a} & {L^{2n+1}_p}}$$ \noindent is a
$\mathbb{Z}_p$-axial map. The conclusion follows from Proposition
\ref{axial}.
\end{proof}

\subsection{Proof of Theorem \ref{Theorem RPn-lens spaces}} \label{Proof-Theorem RPn-lens spaces}

We first establish the equality $\TC(\RP^n)=
\cat(C_{\Delta_{\RP^n}})$ for $n\geq 1$. For $\RP^1=S^1$, $\RP^3$
and $\RP^7$, the result follows from Section \ref{examples
TC=catcofibre}, examples $(ii)$ and $(iii)$. Suppose that $n\neq
1,3,7$. By \cite{F-T-Y} we know that $\TC(\RP^n)\geq n+1$.
Therefore, $\dim(\RP^n)=n \leq \TC(\RP^n)+1-2$ and Theorem
\ref{cofibrediag<=TC} permits us to conclude that $\TC(\RP^n)\geq
\cat(C_{\Delta_{\RP^n}})$. Now suppose that
$\cat(C_{\Delta_{\RP^n}})=k$. Consider the restriction of the
product $m:\RP^{\infty}\times \RP^{\infty}\to \RP^{\infty}$ to
$\RP^n\times \RP^n$. Since $\RP^{\infty}=K(\Z_2,1)$, the map
$m\Delta_{\RP^n}$, which  induces the zero homomorphism in
cohomology with coefficients in $\Z_2$, is homotopically trivial.
Therefore $m$ induces a map $\tilde m:C_{\Delta_{\RP^n}}\to
\RP^{\infty}$. Now, since $\cat(C_{\Delta_{\RP^n}})=k$, $\tilde m$
factors through the $k$-th Ganea fibration of $\RP^{\infty}$ and
hence, by Lemma \ref{ganea} through the inclusion
$j_k:\RP^k\hookrightarrow \RP^{\infty}$. Therefore we obtain the
following homotopy commutative diagram:
 $$\xymatrix{
 & {\RP^n} \ar[d]^{\Delta _{\RP^n}} \\
{\RP^n} \ar[r]^(.4){i_1} & \RP^n\times \RP^n \ar[d]_{\alpha}
\ar@{^{(}->}[r]
& \RP^{\infty}\times \RP^{\infty} \ar[r]^(.6){ } & \RP^{\infty} \\
& C_{\Delta_{\RP^n}}\ar[urr]_{\tilde m } \ar[rr]_{\hat m} && \RP^k
\ar@{^{(}->}[u]_{j_k} }$$ We observe from this diagram that $\hat
m \alpha$ is a $\mathbb{Z}_2$-axial map. As there are not any
$\mathbb{Z}_2$-axial maps $\RP^n\times \RP^n\to \RP^k$ with $k\leq
n$ except for $k=n=1,3$ or $7$ (see Lemma \ref{salvation}) we get
$k\geq n+1.$  Now consider Theorem \ref{Majorations of TC} with
$g=\hat m \alpha .$ Since $\alpha \Delta_{\RP^k}$ is homotopically
trivial then so is $g\Delta_{\RP^k}$. Taking into account that
$\hat m \alpha i_1$ is an $n$-equivalence and $\cat(\RP^k)=k\geq
n+1$, we have $2\dim(\RP^n)\leq r+qk-1$ where $q=1$ and $r=n$.
Therefore $\TC(\RP^n)\leq k.$ This achieves the proof of the part
(a) of the statement.

Now we turn to the case of lens spaces and first establish the
inequality $\mbox{cat}(C_{\Delta _{L^{2n+1}_p}})\leq
\mbox{TC}(L^{2n+1}_p)$. Recall that $L^{2n+1}_p$ is a 0-connected
CW-complex of dimension $2n+1$. On the other hand Farber and Grant
(\cite{F-G}) have proved that $\mbox{TC}(L^{2n+1}_p)\geq 2(k+l)+1$
provided that $0\leq k,l\leq n$ and $p$ does not divide the
binomial coefficient $\binom{k+l}{k}$. We take integers $k,l$ such
that $0\leq k,l\leq n$ and $k+l=n+1$. If $p$ does not divide
$n+1=\binom{n+1}{1}$ we can consider $k=1$ and conclude that
$\mbox{TC}(L^{2n+1}_p)\geq 2(n+1)+1=2n+3$. Suppose that $p$
divides $n+1=\binom{n+1}{1}$. We use the fact that the greatest
common divisor of the binomial coefficients $\binom{n+1}{j}$,
$1\leq j\leq n$, is either a prime $a$ or $1$ according whether
$n+1$ is a power of $a$ or not to conclude that, in our case, this
greatest common divisor must be $1$. Therefore there exists $2\leq
j\leq n-1$ such that $p$ does not divide $\binom{n+1}{j}$ and once
again we can conclude that $\mbox{TC}(L^{2n+1}_p)\geq 2n+3$. Hence
in any case $\mbox{TC}(L^{2n+1}_p)\geq 2n+3$ and
$\dim(L^{2n+1}_p)\leq \mbox{TC}(L^{2n+1}_p)-2$ so, by Theorem
\ref{cofibrediag<=TC} we have that $\mbox{cat}(C_{\Delta
_{L^{2n+1}_p}})\leq \mbox{TC}(L^{2n+1}_p).$

We now prove the other direction. The strategy is the same as for
the case of real projective spaces. We consider the map $\mu':
L^{2n+1}_p\times L^{2n+1}_p\hookrightarrow L^{\infty}_p\times
L^{\infty}_p \stackrel{\mu}{\to} L^{\infty}_p$ where
$\mu(x,y)=xy^{-1}$. The composite $\mu'\Delta_{L^{2n+1}_p}$ is
homotopically trivial and there is a map $\tilde{\mu
}:C_{\Delta_{L^{2n+1}_p}}\rightarrow L^{\infty}_p$ such that
$\tilde{\mu }\alpha $ is homotopic to $\mu'$.
 Now suppose that
$\mbox{cat}(C_{\Delta _{L^{2n+1}_p}})=k$. Then the map $\tilde{\mu
}:C_{\Delta_{L^{2n+1}_p}}\rightarrow L^{\infty}_p$ factors through
the $k$-th Ganea fibration of $L^{\infty}_p$ and hence, by Lemma
\ref{ganea}, through the inclusion
$(L^{\infty}_p)^{(k)}\hookrightarrow L^{\infty}_p$. We have a
homotopy commutative diagram of the following form:
$$\xymatrix{
& L_p^{2n+1}\ar[d]^{\Delta _{L^{2n+1}_p}}\\
{L^{2n+1}_p} \ar[r]^(.4){i_1} &  {L^{2n+1}_p\times L^{2n+1}_p}
\ar[d]_{\alpha} \ar@{^{(}->}[r]
&  {L^{\infty}_p\times L^{\infty}_p} \ar[r]^(.6){\mu } & {L^{\infty}_p}\\
& C_{\Delta_{L^{2n+1}_p}} \ar[urr]_{\tilde \mu } \ar[rr]_{\hat \mu
} && {(L^{\infty}_p)^{(k)}} \ar@{^{(}->}[u]_{j_k} \ar@{^{(}->}[r] &
({L^{\infty}_p})^{(k+1)} \ar@{^{(}->}[ul]
}
$$
We take $g=\hat{\mu }\alpha $. The composite  $g\Delta
_{L^{2n+1}_p}$ is homotopically trivial. There are two
possibilities for $k$. If $k=2r+1$ is odd then
$(L^{\infty}_p)^{(k)}=L_p^{2r+1}$ and by Corollary \ref{axial2} we
have the inequality $r\geq n+1,$ that is $k\geq 2n+3.$ On the
other hand if $k=2r$ is even, then
$(L^{\infty}_p)^{(k+1)}=L_p^{2r+1}$ and by Corollary \ref{axial2}
we obtain $r\geq n+1,$ or $k\geq 2n+2.$ In any case we have $k\geq
2n+2$. As a consequence the inclusion $j_k$ is at least a
$(2n+1)$-equivalence and hence the map $f=gi_1$ is a
$(2n+1)$-equivalence since the top row is homotopic to the
inclusion $j_n:L^{2n+1}_p\hookrightarrow L^{\infty }_p,$ which is
a $(2n+1)$-equivalence. Since $2\dim (L_p^{2n+1})=4n+2\leq
(2n+1)+(2n+2)-1$ we conclude, by Theorem \ref{Majorations of TC}, that
$\TC(L_p^{2n+1})\leq
\cat((L^{\infty}_p)^{(k)})=k=\mbox{cat}(C_{\Delta
_{L^{2n+1}_p}})$.

%\begin{remark}{\rm
%Note that in the proof of the above theorem, the integer
%$k=\mbox{cat}(C_{\Delta _n})$ satisfies $k>3.$ Indeed, suppose
%that $k\leq 3;$ then there is a similar homotopy commutative
%diagram
%$$\xymatrix{
%{L^{2n+1}_p\times L^{2n+1}_p} \ar[d]_{\alpha} \ar@{^{(}->}[r]
%& {L^{\infty}_p\times L^{\infty}_p} \ar[r]^(.6){\mu } & {L^{\infty}_p}\\
% C_{\Delta_{L^{2n+1}_p}} \ar[urr]_{\tilde \mu } \ar[rr]_{\hat \mu } &&
%{(L^{\infty}_p)^{(k)}} \ar@{>>}[u] \ar@{^{(}->}[r] & {L^3_p}
%\ar@{^{(}->}[ul]_{j_1} }$$ By Corollary \ref{axial2}, $1\geq n+1$
%(or $n\leq 0$), which is impossible. }
%\end{remark}
%

\end{document}